\documentclass[a4paper, 11pt]{amsart}

\usepackage{amsmath, amssymb, amsthm}%, wasysym}
\usepackage[english]{babel}
\usepackage[T1]{fontenc}
\usepackage{lmodern}
\usepackage{graphicx}
\usepackage{url}            % simple URL typesetting
\usepackage{booktabs}       % professional-quality tables
\usepackage{amsfonts}       % blackboard math symbols
\usepackage{nicefrac}       % compact symbols for 1/2, etc.
\usepackage{microtype}      % microtypography
\usepackage{lipsum}
\usepackage{amsmath}
\usepackage{amssymb}
\usepackage{xcolor}
\usepackage{bm}
\usepackage{cite}
\theoremstyle{remark}

\theoremstyle{definition}

\newtheorem{theorem}{Theorem}

\newtheorem{proposition}{Proposition}
\newtheorem{corollary}{Corollary}
\newtheorem{lemma}[theorem]{Lemma}
\newtheorem{remark}{Remark}

\newtheorem{definition}{Definition}

\def \k{\bm{k}}
\def \v{\bm{v}}
\def \w{\bm{w}}
\def \x{\bm{x}}
\def \y{\bm{y}}

\def\vett#1{\bm{#1}}

\def\lip{Lipschitz}

\def\mi{\hat{\mu}( \k)} %definisce la mu trasformata
\def \ni{\hat{\nu}(\k)}
\def \fspa{f^{(\alpha)}_{s,p}}

\def \supk{\sup_{\k \in \mathbb{R}^d\backslash \{0\}}}

\def \intTnormd{\dfrac{1}{|T|^d}\int_{[0,T]^d}}
\def \intnormpiN{\dfrac{1}{|T|^d}\int_{[0,T]^d}}
\def \intaTnormd{\dfrac{1}{|\gamma T|^d}\int_{[0,\gamma T]^d}}
\def \intaTnormdue{\dfrac{1}{|T|^2}\int_{[0,T]^2}}
\def \intaTnormdued{\dfrac{1}{|T|^d}\int_{[0,T]^d}}
\newcommand{\fer}[1]{(\ref{#1})}
\def\erre{\mathbb{R}}
\def\R{\mathbb{R}}

\def\probRN{\mathcal{P}(\erre^d)}

\def\probRNm{\mathcal{P}_m(\R^d)}
\def\borelN{\mathcal{B}(\erre^d)}
\def\probRNN{\mathcal{P}(\erre^d\times\erre^d)}

\def\normadue#1{|#1|}

\def\interrenn#1{\int_{\erre^{d}\times \erre^{d}}#1}
\def\planmunu{\Pi(\mu,\nu)}

\def\mmu{\bm{m}_{\mu}}
\def\nnu{\bm{m}_{\nu}}

\def\ttrasl#1{#1_{\bm{\tau}}}
\def\vtrasl#1{#1_{\bm{v}}}
\def\wtrasl#1{#1_{\bm{w}}}

\begin{document}

\title[The equivalence of Fourier-based and Wasserstein metrics]{The Equivalence of Fourier-based and Wasserstein Metrics on Imaging Problems}

\author{G. Auricchio, A. Codegoni, S. Gualandi, G. Toscani, M. Veneroni}
\address{Department of Mathematics, University of Pavia, 
via Ferrata 1,
Pavia, 27100 Italy}
\email{gennaro.auricchio01@universitadipavia.it, andrea.codegoni01@universitadipavia.it, stefano.gualandi@unipv.it, giuseppe.toscani@unipv.it, marco.veneroni@unipv.it.}

\maketitle

\begin{abstract}
We investigate properties of some extensions of a class of Fourier-based probability metrics, originally introduced to study convergence to equilibrium for the solution to the spatially homogeneous Boltzmann equation. 
At difference with the original one, the new Fourier-based metrics are well-defined also for probability distributions with different centers of mass, and for discrete probability measures supported over a regular grid. 
Among other properties, it is shown that, in the discrete setting, these new Fourier-based metrics are equivalent either to the Euclidean-Wasserstein distance $W_2$, or to the Kantorovich-Wasserstein distance $W_1$, with explicit constants of equivalence. Numerical results  then show that in benchmark problems of image processing,  Fourier metrics provide a better runtime with respect to Wasserstein ones.  
\end{abstract}
\vskip 3mm

\keywords{Fourier-based Metrics; Wasserstein Distance; Fourier Transform; Image analysis.}
\vskip 2mm
{ AMS Subject Classification: 60A10, 60E15, 42A38, 68U10.}

%\vskip 2cm
\section{Introduction}

In computational applied mathematics, numerical methods based on Wasserstein distances achieved a leading role over the last years. Examples include the comparison of histograms in higher dimensions \cite{ling2007efficient,auricchio2018computing,bassetti2018computation},  image retrieval \cite{rubner2000earth}, image registration \cite{haker2004optimal,Bonneel2019}, or, more recently, the computations of barycenters among images \cite{auricchio2019computing,cuturi2014fast}. 
Surprisingly, the possibility to identify the cost function in a Wasserstein distance, together with the possibility of representing images as histograms, led to the definition of classifiers able to mimic the human eye \cite{rubner2000earth,pele2009fast,peyre2019computational}.

More recently, metrics which are able to compare at best probability distributions were introduced and studied in connection with  machine learning, where testing the efficiency of new classes of loss functions for neural networks training has become increasingly important.  In this area, the Wasserstein distance often turns out to be the appropriate tool \cite{arjovsky2017wasserstein,adler2018banach,frogner2015learning}. Its main drawback, though, is that it suffers from high computational complexity. For this reason, attempts to use other metrics, which require a lower computational cost while maintaining a good approximation, have been object of recent research \cite{wavelet}. There,  the theory of approximation in the space of wavelets was the main mathematical tool.

Following the line of thought of \cite{wavelet}, we consider here an alternative to the approximation in terms of wavelets, which is furnished by metrics based on the Fourier transform. In terms of computational complexity, the price to pay for  a dimension $N\gg 1$ of the data changes  from a time $O(N)$ to the time $O(N \log N)$ required to evaluate the fast Fourier transform.

While this represents a worsening, with respect to the use of wavelets, in terms of computational complexity, there is an effective improvement with respect to the computational complexity required to evaluate Wasserstein-type metrics, which is of the order $O(N^3\log N)$. Furthermore, from the point of view of the important questions related to the comparison of these metrics with Wasserstein metrics in problems motivated by real applications, we prove in this paper that in the case of probability measures supported on a bounded domain, one has a precise and explicit evaluation of the constants of equivalence among these Fourier-based metrics and the Wassertein ones, a result which is not present in \cite{wavelet}. 

The Fourier-based metrics considered in this paper were introduced in \cite{gabetta1995metrics}, in connection with the study of the trend to equilibrium for solutions of the spatially homogeneous Boltzmann equation for Maxwell molecules. 
Since then, many applications of these metrics have followed in both kinetic theory and  probability \cite{toscani1999probability,pulvirenti2004asymptotic,carlen1999propagation,
goudon2002fourier,bisi2006decay,carlen2000central,Toscani2007}. All these problems deal with functions supported on the whole space $\erre^d$, with $d \ge 1$, that exhibit a suitable decay at infinity which guarantees the existence of a suitable number of moments.

Given two probability measures $\mu,\nu\in\mathcal{P}(\erre^d)$, $ d \ge 1$, and a real parameter $s>0$, the Fourier-based metrics $d_s$ considered in \cite{gabetta1995metrics} are given by
\begin{equation}\label{fm}
    d_s(\mu,\nu) := \sup_{\k\in \erre^d\setminus\{0\}} \frac{|\widehat{\mu}(\k)-\widehat{\nu}(\k)|}{|\k|^s},
\end{equation}
where $\widehat{\mu}$ and $\widehat{\nu}$ are the Fourier transforms of the measures $\mu$ and $\nu$, respectively. As usual, given a probability measure $\mu \in \probRN$, the Fourier transform of $\mu$ is defined by
    \[
    \hat{\mu}(\k):=\int_{\mathbb{R}^d}e^{-i\k\cdot\x}d\mu(\x).
    \] 
These metrics, for $s \ge 1$, are well-defined under the further assumption of boundedness  and equality of some moments of the probability measures. Indeed, a necessary condition for $d_s$ to be finite, is that moments up to $[s]$ (the integer part of $s$) are equal for both measures  \cite{gabetta1995metrics}.

In dimension $d=1$, similar metrics were introduced a few years later by Baringhaus and Gr\"ubel in connection with the characterization of convex combinations of random variables \cite{Baringhaus1997}. Given two probability measures $\mu,\nu\in\mathcal{P}(\erre^d)$, $ d \ge 1$, and two  real parameters $s>0$ and $p \ge 1$, the multi-dimensional version of these Fourier-based metrics reads
\begin{equation}\label{baring}
D_{s,p}(\mu,\nu):=\left(\int_{\erre^d}\dfrac{|\widehat{\mu}(\k)-\widehat{\nu}
(\k)|^p}{|\k|^{(ps+d)}}d\k \right)^{{1}/{p}}.
\end{equation}
The metrics defined by \fer{fm} and \fer{baring} belong to the set of ideal metrics \cite{zolotarev}, and have been shown to be equivalent to other common probability distances \cite{gabetta1995metrics,toscani1999probability}, including the Wasserstein distance $W_2(\mu,\nu)$ \cite{Toscani2007}, given by
    \begin{equation}\label{w2}
    W_2(\mu,\nu):=\inf_{\pi \in \Pi(\mu,\nu)}\bigg\{ \interrenn |\bm x-\bm y|^2\, d\pi(\bm x,\bm y) \bigg\}^{1/2},
    \end{equation} 
where the infimum is taken on the set $ \Pi(\mu,\nu)$  of all probability measures on $\erre^d \times \erre^d$ with marginal densities $\mu$ and $\nu$.  However, in dimension $d>1$ the constants of equivalence are not explicit \cite{Toscani2007}, so that it is difficult to establish a comparison between these metrics' efficacy in applications.
 
An unpleasant aspect related to the application of the previous Fourier-based distances is related to its finiteness, that requires, for high values of $s$, a sufficiently high number of equal moments for the underlying probability measures. In the context of kinetic equations of Boltzmann type, where conservation of momentum and energy of the solution is a consequence of the microscopic conservation laws of binary interactions among particles, this requirement on $d_s$, with $2<s<3$, is clearly not restrictive.  However, in order to apply the Fourier-based metrics outside of the context of kinetic equations, this requirement appears unnatural. To clarify this point, let us consider the case in which we want to compare the distance between two images. If we take two grey scale images and model them as probability distributions, there is no reason why these distributions possess the same expected value. The simplest example is furnished by two images consisting of a black dot, each one centered in a different point of the region, that can be modeled as two Dirac delta functions centered in two different points.

In this paper we improve the existing results concerning the evaluation of the constants in the equivalence relations between the Fourier-based metrics and the Wasserstein one, in a relevant setting with respect to applications. This equivalence is related to the comparison of two discrete measures and it is based on the properties of the Fourier transform in the discrete setting. To this extent, we consider a new version of these metrics, the \emph{periodic Fourier-based metrics}, that play the role of the metrics \fer{fm} and \fer{baring} in the discrete setting.
With our results, we show that the new family of Fourier-based metrics represents a fruitful alternative  to  the Wasserstein metrics, both from the theoretical and the computational points of view.

To weaken the restriction about moments, we further consider a variant of the Fourier metric $d_2$ that remains well-defined even for probability measures with different mean values.

The content of this paper is as follows. In Section \ref{recallot} we introduce the notations and the basic concepts of measure theory and optimal transport. Furthermore,  we define the Fourier-based metrics, we recall their main properties, and we introduce our extension. Then, in view of applications, in Section \ref{discreto}, we consider a discrete setting and we define and study the properties of the new family of  {periodic Fourier-based metrics}, highlighting their explicit equivalence with the Wasserstein distance in various cases. Section \ref{Numericalresults} presents numerical results obtained comparing our implementation of the \emph{periodic Fourier-based metrics} with the Wasserstein metrics as implemented in the POT library \cite{flamary2017pot}. The concluding remarks are contained in Section \ref{fine}. 

%%%%%%%%%%%%%%%%%%%%%%%%%%%%%%%%%%%%%%%%%%%%%%%%%%%%%%%%%%%%%%%%%%%%%%%%%%%%%
\section{An extension of Fourier-based metrics}
\label{recallot}
In what follows, we briefly review some basic notions of optimal transport, together with the definition and some properties of Wassertein and  Fourier-based metrics. The final goal is to extend the definition of the metrics \fer{fm} and \fer{baring} for the particular case $s=2$, which allows for a direct and fruitful comparison between the Fourier-based metrics and the Wasserstein metric $W_2$ defined in \fer{w2}. In what follows, we only present  the notions that are necessary for our purpose. For a deeper insight on optimal transport, we refer the reader to \cite{villani2008oan,santambrogio2015optimal,ambrosio2008gradient,ambrosio2013user}. Likewise, we address the interested reader to  \cite{Toscani2007} for an exhaustive review of the properties  of the Fourier-based metrics and their connections with other metrics used in probability theory. 

We work on the Euclidean space $\R^d$, endowed with the Borel $\sigma-$algebra $\borelN$. We use bold letters to denote vectors of $\R^d$. If $\bm x \in \R^d$, then $x_i$ denotes its $i$-th coordinate. Given $\bm x, \bm y \in \R^d$, $\langle \bm x,\bm y\rangle = \sum_{i=1}^n x_iy_i$ is their scalar product and $|\bm x| = \langle \bm x,\bm x\rangle^{1/2} $  is the Euclidean norm (or modulus) of $\bm x$.

The set of probability measures on $\R^d$ is denoted by $\probRN$. 
%The mean value of $\mu \in \probRN$ is
%\[
%	\speranza{\mu}=\int_{\R^d} \bm x\, d\mu(\bm x).
%\]
For all $m\in \mathbb{N}$ we denote by $\probRNm$ the set of probability measures with finite moments up to order $m$
\[    \probRNm:= \bigg\{ \mu\in \probRN :
    \int_{\R^d} \bm x^\beta \,d\mu(\bm{x}) < +\infty,\ 
    \forall \beta\in \mathbb{N}^d,\ |\beta|\leq m 
    \bigg\}.
\]
Given $\mu \in \probRN$ and a Borel  map $f:\R^d \to \R^d$, then the image measure (or push-forward) of $\mu$ by $f$ is $f_\# \mu \in \probRN$, given by $f_\# \mu (A) = \mu ( f^{-1}(A))$ for  all $A \in \borelN$. Equivalently, for every continuous compactly supported function $\phi$ on $\R^d$, it holds
\[
	\int_{\R^d}\phi(y)\,d(f_\# \mu)(y) =  \int_{\R^d}\phi(f(x))\,d\mu(x).
\]
 
Our first goal is to define the Fourier-based metrics $d_s$, in the range $1 < s \le 2$, on $\probRN$.

\begin{definition}
Given $\mu \in \mathcal{P}_1(\R^d)$, we say that
\[
	\mmu =  \int_{\R^d} \bm x\, d\mu(\bm x)
\]
is the \emph{center} of $\mu$.
\end{definition}

The center of a measure $\mu$ can be moved by resorting to a translation. Given $\mu\in \mathcal{P}_1(\R^d)$ and $\bm{\tau} \in \R^d$, we define the translated measure $\mu_{\bm\tau} \in \mathcal{P}_1(\R^d)$ by
\[
	\mu_{\bm\tau} = S^{\bm \tau}_\# \mu,\quad \text{where}\quad S^{\bm \tau}(\bm x)=\bm x + \bm \tau.
\]
%   mi sembra piu corretto definire una misura sui Boreliani, piuttosto che punto per punto, dato che non si tratta di una funzione. Ho anche tolto T_\tau dal gioco, dato che non compare quasi mai e usate sempre \mu_\tau
%
% operator $T_{\bm\tau}:\probRN \rightarrow \probRN$ is defined as
%    \begin{equation}\label{operatoretraslazione}
%    T_{\bm\tau}(\mu) = \mu_{\bm\tau},
%    \end{equation}
%    where the measure $\mu_{\bm\tau}$, for each vector $\bm x \in \erre^d$  is given by
%    \[
%    \mu_{\bm\tau} (\{\bm x\})=\mu(\{\bm{x-\tau}\}).
%    \]

%\begin{definition}
%   A measure $\mu$ is \emph{centered} in $\bm a \in \mathbb{R}^d$, if its expected value is $\bm a$, that is
%   \[
%   \mathbb{E}(\mu):=\interren \bm x\ d\mu(\bm x)=\bm a.
%   \]
%   If the previous expression holds, we say that $\bm a$ is the center of the measure $\mu$.
%\end{definition}

%It is immediate to show that, given two probability measures with centers located in $\bm a$, (respectively $\bm b$), it is always possible to center the two measures in the same point by translation. 

\begin{lemma}
    \label{momentiuguali}
    Given $\mu, \nu \in \mathcal{P}_1(\R^d)$, there  exists a unique vector $\bm\tau \in \R^d$ such that
    \[
    		\mmu=\bm{m}_{\nu_{\bm\tau}}.
    \]
 \end{lemma}
\begin{proof}
    Let $\bm\tau=\mmu-\nnu$, then 
    \[
    		\bm{m}_{\nu_{\bm\tau}} = \int_{\R^d}\bm x d\ttrasl{\nu}(\bm x) 
    =\int_{\R^d}(\bm{x+\tau})d\nu(\bm x)=\nnu+\bm \tau=\mmu.
    \]%{}
\end{proof}{}

Let us recall now the definition of transport plan, and the consequent definition of Wasserstein Distance.

\begin{definition}[Transport plan]
    Given two probability measures $\mu,\nu\in \probRN$, a vector $\pi \in \probRNN$ is called a transport plan between $\mu$ and $\nu$ if its marginals coincide with $\mu,\nu$, that is  
    \begin{eqnarray}
    \pi(A\times \erre^d) &=& \mu(A) \quad \quad \forall A \in \borelN,\\
    \pi(\erre^d\times B) &=& \nu(B) \quad \quad \forall B \in \borelN.
    \end{eqnarray}
We denote by $\planmunu$ the set of all transport plans between $\mu$ and $\nu$.
\end{definition}

\begin{definition}[Wasserstein distance]
    Given $p\in\mathbb{N}$ and $\mu,\nu \in \mathcal{P}_p(\R^d)$, the Wasserstein distance of order $p$ between $\mu$ and $\nu$ is defined as
    \begin{equation}
    W_p(\mu,\nu):=\inf_{\pi \in \Pi(\mu,\nu)}\bigg\{ \interrenn |\bm x-\bm y|^p\, d\pi(\bm x,\bm y) \bigg\}^{1/p},
    \label{wass}
    \end{equation}
where $|\cdot|$ is a norm defined in $\erre^d$.
\end{definition}

In this paper, we consider only the Euclidean norm, and we focus on Wasserstein distances with exponents $p=1$ and $p=2$, namely
\begin{align}
W_1(\mu,\nu) & :=\inf_{\pi \in \Pi(\mu,\nu)}\bigg\{ \interrenn |\bm x-\bm y|\, d\pi(\bm x,\bm y) \bigg\},\\
\label{eq:W2} W_2(\mu,\nu) & :=\inf_{\pi \in \Pi(\mu,\nu)}\bigg\{ \interrenn |\bm x-\bm y|^2\, d\pi(\bm x,\bm y) \bigg\}^{1/2}.
\end{align}
The $W_2$ metric satisfies an explicit translation property  (Remark 2.19, \cite{peyre2019computational} ). We give below a short proof of this property.

\begin{lemma}
    \label{lemmaWtrasl}
    Let $\mu,\nu \in \mathcal{P}_2(\R^d)$, with centers $\mmu$ and $\nnu$, respectively. For any given pair of vectors $\bm{v},\bm{w} \in \R^d$  we have 
    \begin{equation}
    W_2(\vtrasl{\mu},\wtrasl{\nu})^2=W_2(\mu,\nu)^2+\normadue{\bm v-\bm w}^2+2\langle\bm v-\bm w,\mmu-\nnu \rangle.
    \end{equation}
    In addition, if we choose $\bm v=-\mmu$ and $\bm w=-\nnu$ it holds
    \begin{equation}
    W_2(\mu_{-\mmu},\nu_{-\nnu})^2 = W_2(\mu,\nu)^2 - \normadue{\mmu -\nnu}^2.
    \end{equation}
\end{lemma}

\begin{proof}
    Given a transport plan $\pi \in \planmunu$, we consider the transport plan
    \[
    		\tilde{\pi}:=(S^{\bm v},S^{\bm w})_\#\pi,
    \]
    where $S^{\bm v}(\bm x)=\bm x + \bm v$, $S^{\bm w}(\bm y)=\bm y + \bm w$.  $\tilde \pi$ is a transport plan between the translated measures $\mu_{\bm{v}}$ and $\nu_{\bm{w}}$. Then, by definition of push-forward, we get
    \begin{eqnarray*}
    &&\hspace{-1cm}\interrenn{|\x-\y|^2d\tilde{\pi}(\x,\y)}\\
%    &=&\interrenn{|\x-\y|^2d{\pi}(\x-\bm v,\y-\bm w)}\\
    &=&\interrenn{|(\x+\bm v)-(\y+\bm w)|^2d{\pi}(\x,\y)}\\
    &=&\interrenn{(|\x-\y|^2+|\bm v-\bm w|^2+2\langle\x-\y,\bm v-\bm w\rangle)d{\pi}(\x,\y)}\\
    %&=&\interrenn{|\x-\y|^2d{\pi}(\x,\y)} + \interrenn{|\bm v-\bm w|^2d{\pi}(\x,\y)}\\
    %&&\quad +2\interrenn{\langle\x-\y,\bm v-\bm w\rangle d{\pi}(\x,\y)}\\
    &=& \interrenn{|\x-\y|^2d{\pi}(\x,\y)} +
    |\bm v-\bm w|^2+
    2\langle \mmu -\nnu,\bm v-\bm w\rangle.
    \end{eqnarray*}
    
    \noindent If $\pi$ is an optimal transport plan between $\mu$ and $\nu$, we have
    \begin{eqnarray*}
    \nonumber W_2(\vtrasl{\mu},\wtrasl{\nu})^2&\leq&\interrenn{|\x-\y|^2d\tilde{\pi}(\x,\y)}\\
    &=&W_2(\mu,\nu)^2+|\bm v-\bm w|^2+2\langle\bm v-\bm w,\mmu -\nnu\rangle.
    \end{eqnarray*}
    By repeating the previous argument with an optimal transport plan between $\vtrasl{\mu},\ \wtrasl{\nu}$, we find
    \begin{align*}
    W_2(\vtrasl{\mu},\wtrasl{\nu})^2 &=\interrenn{|\x-\y|^2d\pi(\x,\y)}+|\bm v-\bm w|^2
    	+2\langle\bm v-\bm w,\mmu -\nnu\rangle\\
		& \geq W_2(\mu,\nu)^2+|\bm v-\bm w|^2+2\langle\bm v-\bm w,\mmu -\nnu\rangle.
    \end{align*}
    Hence, we can conclude
    \begin{equation*}
    W_2(\vtrasl{\mu},\wtrasl{\nu})^2=W_2(\mu,\nu)^2+|\bm v-\bm w|^2+2\langle\bm v-\bm w,\mmu -\nnu\rangle.
    \end{equation*}
    {}
\end{proof}

The idea of using translation operators to compute the distance of probability measures with different centers can be used to properly modify the Fourier-based metrics $d_s$ and $D_{s,p}$ defined in \fer{fm} and \fer{baring}. Indeed, as briefly discussed in the introduction, the case $s\ge 1$ requires the probability measures to satisfy the further condition given below \cite{gabetta1995metrics}.
\begin{proposition}[Proposition 2.6, \cite{Toscani2007}]
    \label{finitezza}
Let $\lfloor s \rfloor$ denote the integer part of $s\in\erre$, and assume that the densities $\mu,\nu \in \mathcal{P}_s(\erre^d)$ possess equal moments up to $\lfloor s\rfloor$ if $s\notin \mathbb{N}$,    or equal moments up to $s-1$ if $s\in\mathbb{N}$. Then the Fourier-based distance $d_s(\mu,\nu)$ is well-defined. In particular, $d_2(\mu,\nu)$ is well-defined for two densities with the same center.
\end{proposition}
The interest in the $d_2$ metric is related to its equivalence  to the Euclidean Wasserstein distance $W_2$. A detailed proof in dimension $d\ge1$ can be found in the review paper \cite{Toscani2007}.

%In the next section we extend the Fourier metric $d_2$ in order to overcome the request on the expected value stated in Proposition \ref{finitezza} when $s=2$. The property of the $d_2$ on which we focus on and which we want to preserve with our extension is
\begin{theorem}[Proposition 2.12 and Corollary 2.17, \cite{Toscani2007}]
\label{equi1}
For any given pair of probability densities $\mu,\nu \in \mathcal{P}_2(\erre^d)$  such that $\mmu=\nnu$, 
    the $d_2$ metric is equivalent to the Euclidean Wasserstein distance $W_2$, that is, there exist two positive bounded constants $c <C$ such that
    \begin{equation}
    \label{eqWd2}
    cW_2(\mu,\nu)\leq d_2(\mu,\nu) \leq CW_2(\mu,\nu).
    \end{equation}
    %Furthermore, the upper bound constant is 
    %\[
    %C=
    %\Bigg(\sqrt{\interren{|\x|^2d\mu(\x)}} %+\sqrt{\interren{|\x|^2d\nu(\x)}}\Bigg).
    %\]
    \end{theorem}
The proof in \cite{Toscani2007} does not provide in general the explicit expression of the two constants $c$ and $C$. The value of these constants is quite involved, and it is strongly dependent on higher moments of the densities. 

The equivalence result of Theorem \ref{equi1} can easily be extended to cover  the case of probability measures with different centers of mass. To this aim it is necessary, in analogy with the property  of Wasserstein distance $W_2$  stated in Lemma \ref{lemmaWtrasl},  to modify the Fourier-based metrics $d_2$ and $D_{2,p}$ in such a way to allow for probability measures with different centers of mass. We start by considering the case of the metric $d_2$.

\begin{definition}[Translated Fourier-based Metric]
    We define the function $\mathcal{D}_2: \mathcal{P}_2(\R^d)\times \mathcal{P}_2(\R^d) \to \R$ as:
    \begin{equation}
    \label{extFou}
    		\mathcal{D}_2(\mu,\nu) := \sqrt{d_2(\mu,\nu_{{\bm m}_\mu -{\bm m}_\nu})^2 + |\mmu -\nnu|^2 }.
    \end{equation}

\end{definition}

Owing to Remark \ref{momentiuguali} and Proposition \ref{finitezza}, $\mathcal{D}_2(\mu,\nu)$ is well-defined for each pair of probability measures in $\mathcal{P}_2(\R^d)$, independently of their centers. 
Note that $\nu_{\mmu -\nnu}$, which is the translation of $\nu$ by $\mmu -\nnu$, has the same center as $\mu$. 
One could give an equivalent definition of $\mathcal{D}_2$ by translating $\mu$, instead of $\nu$, or by translating both centers to $\bm 0$.
%However, the value of $\mathcal{D}_2$ does not depend on which translation we operate on one or both the measures, as soon as, after translations,  they end up with the same center of mass. Indeed it holds
\begin{lemma}
    \label{lemmatranslate}
    Given $\mu,\nu \in \mathcal{P}_2(\R^d)$ and $\v,\w \in \mathbb{R}^d$, then
    \[
    		|\widehat{\vtrasl{\mu}}(\k)-\widehat{\wtrasl{\nu}}(\k)|=|\hat{\mu}(\k)-\widehat{\nu_{\bm w-\bm v}}(\k)| 
    = |\widehat{\mu_{\bm v-\bm w}}(\k)-\hat{\nu}(\k)|.
    \]
    Therefore
    \begin{equation*}
    d_2(\vtrasl{\mu},\wtrasl{\nu})=d_2({\mu},\nu_{\bm w- \bm v})=d_2(\mu_{\bm v- \bm w},\nu).
    \end{equation*}
    In particular, the  function $
    (\mu,\nu) \to d_2({\mu},\nu_{\mmu -\nnu})$ is symmetric.
\end{lemma}

\begin{proof}
    By the translation property of the Fourier Transform, for all $\bm v \in \R^d$ we have the identity
\[
 	\widehat{\mu_{\bm v}}(\k) = e^{-i{\bm v}\cdot\k}\hat{\mu}(\k).
 \]
 Therefore
    \begin{eqnarray*}
    |e^{-i\v\cdot\k}\hat{\mu}(\k)-e^{-i\w\cdot\k}\hat{\nu}(\k)|
    &=& |e^{-i\w\cdot\k}(e^{-i(\v-\w)\cdot\k}\hat{\mu}(\k)-\hat{\nu}(\k))| \\
    &=& |e^{-i(\bm v-\bm w)\cdot \k}\hat{\mu}(\k)-\hat{\nu}(\k)|.
    \end{eqnarray*}
 This shows that
    \begin{equation*}
    \supk\frac{|e^{-i\v\cdot\k}\hat{\mu}(\k)-e^{-i\w\cdot\k}\hat{\nu}(\k)|}{|\k|^2}=\supk\frac{|e^{-i(\bm v - \bm w) \cdot\k}\hat{\mu}(\k)-\hat{\nu}(\k)|}{|\k|^2}. 
    \end{equation*}
    
\end{proof}

Lemma \ref{lemmatranslate} implies the following theorem.
\begin{theorem}
\label{teoremadistanza}
    The function $\mathcal{D}_2$ defined in \eqref{extFou} is a distance over $\mathcal{P}_2(\erre^d)$.
\end{theorem}

\begin{proof}
Clearly $\mathcal{D}_2(\mu,\nu)\geq 0, \forall\mu,\nu\in\mathcal{P}_2(\erre^d)$, and   $\mathcal{D}_2(\mu,\nu)= 0$ if and only if $\mu = \nu$.
  Symmetry  follows from %the fact that $|\Delta\mathbb{E}|$ is symmetric by definition, while the symmetry of $d_2$ is proved in 
  Lemma \ref{lemmatranslate}. %Thus, in order to show that $\mathcal{D}_2$ is a distance, it remains to prove the triangle inequality. This follows since 
  Finally, both $d_2(\mu,\nu) $, in reason of the fact that it is a distance, and $|{\bm m}_\mu -{\bm m}_\nu|$  satisfy the triangular inequality.
  
    {}
\end{proof}

An analogous extension can be done for the metric $D_{2,p}$ defined in \fer{baring}.
\begin{definition}
    Given  $p\ge 1$,  we define $\mathcal{D}_{2,p}:  \mathcal{P}_2(\R^d)\times \mathcal{P}_2(\R^d) \to \R$ by 
   \[
        \mathcal{D}_{2,p}(\mu,\nu) := \sqrt{D_{2,p}(\mu,\nu_{\mmu-\nnu})^2 + |\mmu-\nnu|^2 }.
    \]
  $\mathcal{D}_{2,p}$ is a metric on $\mathcal{P}_2(\R^d)$.
\end{definition}

It is remarkable that the result of Theorem \ref{equi1} can be extended to  the $\mathcal{D}_2$ metric. %Indeed we have
\begin{theorem}
\label{teoremaequivalenza}
    The function $\mathcal{D}_2$ defined in \eqref{extFou} is equivalent to the $W_2$ distance.
\end{theorem}

\begin{proof}
    Let $\mu,\nu\in\mathcal{P}_2(\erre^d)$ and let $\mu^*,\nu^*$ denote the two corresponding translated measures centered in $\vett 0$. By Lemma \ref{lemmaWtrasl}, we have
    \begin{equation}
        \label{tform}
        W_2^2(\mu,\nu)=W_2^2(\mu^*,\nu^*)+|\mmu-\nnu|^2.
    \end{equation}
    Owing to Theorem \ref{equi1}, there exist two constants $c,C \in (0,\infty)$ such that
    \begin{equation}
        \label{toscaniform}
        cd_2(\mu^*,\nu^*)\leq W_2(\mu^*,\nu^*)\leq Cd_2(\mu^*,\nu^*).
    \end{equation}
    Using (\ref{tform}) in (\ref{toscaniform}), we get
    \begin{equation*}
        cd_2(\mu^*,\nu^*)^2+|\mmu-\nnu|^2\leq W_2(\mu,\nu)^2\leq Cd_2(\mu^*,\nu^*)^2+|\mmu-\nnu|^2,
    \end{equation*}
    which can be rewritten as
    \begin{multline}
        \min\{c,1\}\big(d_2(\mu^*,\nu^*)^2+|\mmu-\nnu|^2\big)\leq W_2(\mu,\nu)^2\\\leq \max\{1,C\}\big(d_2(\mu^*,\nu^*)^2+|\mmu-\nnu|^2\big)\nonumber.
    \end{multline}
    Finally
    \[
       \min\{c,1\} \, \mathcal D_2^2(\mu,\nu) \leq W_2^2(\mu,\nu) \leq \max\{1,C\} \, \mathcal D_2^2(\mu,\nu).
    \]
    {}
\end{proof}

%%%%%%%%%%%%%%%%%%%%%%%%%%%%%%%%%%%%%%%%%%%%%%%%%%%%%%%%%%%%%%%%%%%%%%%%%%%%%%%%%%%%
\section{The Periodic Fourier-based metrics}
\label{discreto}
In this section,  we introduce a family of (Discrete) Periodic Fourier-based metrics suitable to measure the distance between discrete probability measures whose support is restricted to a given set of  points, and we discuss their equivalence with the Wasserstein metrics. The main result is that in this case one obtains a precise estimation of the constants of equivalence.
%\subsection{Discrete setting}
\begin{definition}[Regular grid]
%   Let us take the $d$-dimensional hypercube in $\erre^d$
%   \[
%   C_{d}:=[0,1]^d.
%   \]
For $N\in \mathbb N \setminus \{0\}$,  we define the regular grid 
\[
	G_N:= \left\{ \bm x \in \R^d : N\bm x \in \mathbb Z^d \cap [0,N)^d \right\}.
\]	
% $G_N$ be a regular grid composed of the following $N^d$ points:
%\begin{equation*}
%    G_N := \big\{ \bm x_\beta \big\} \quad \mbox{ with } \quad \bm x_\beta =  \left(\frac{i_1-1}{N},\frac{i_2-1}{N},\dots,\frac{i_d-1}{N} \right),
%\end{equation*}
%where $\beta := (i_1, \dots, i_d) \in \{1, \dots, N\}^d$. 
Note that $G_N \subset [0,1)^d$.
%   G_{N}:=\bigg\{ \bigg(\dfrac{\alpha_{1}}{N},\dots, \dfrac{\alpha_{d}}{N} \bigg) \bigg\}_{\alpha_1=0,\dots,N-1;\dots;\alpha_d=0,\dots,N-1}.
%   \]
\end{definition}

\begin{definition}[Discrete Measure over a grid]
    We say that $\mu$ is a a discrete measure over $G_N$ if its support is contained in $G_N$, that is, if $\mu$ has the form
    \begin{equation} 
    \label{discretemeasure}
    		\mu(\bm x)=\sum_{\bm y \in G_N}\mu_{\bm y}\delta({\bm x - \bm y}), 
    \end{equation}
\end{definition}
where $\mu_{\bm y} \in \R, \mu_{\bm y} \geq 0$ for all $\bm y \in G_N$.

The Discrete Fourier transform of a discrete measure over $G_N$ is given by
%\begin{eqnarray}
%    \nonumber\mi&=&\sum_{\bm x_\beta\in G_N}\mu_{\bm x_\beta}e^{-i\langle \bm x_\beta, \bm k\rangle}
%    = \sum_{\bm x_\beta\in G_N}\mu_{\bm x_\beta}\big( \cos{(\langle \bm x_\beta, \bm k\rangle)} - i \sin{(\langle \bm x_\beta, \bm k\rangle)} \big), \\
%    \nonumber&=& \sum_{\bm x_\beta\in G_N}\mu_{\bm x_\beta} \bigg( \cos\bigg({k_1 \frac{i_1-1}{N} +\dots +k_d \frac{i_d-1}{N} \bigg)} \\
%    \label{periodicity} && \hspace{2cm} - i \sin{\bigg(k_1 \frac{i_1-1}{N} +\dots +k_d \frac{i_d-1}{N} \bigg)} \bigg).
%\end{eqnarray}
\begin{equation}
    \label{periodicity}
    \mi = \sum_{\bm x \in G_N}\mu_{\bm x}e^{-i\langle \bm x, \bm k\rangle}.
\end{equation}
The %equation \eqref{periodicity} 
periodicity of the complex exponential implies that $\hat{\mu}$ is $2\pi N$-periodic over all directions, so that it is sufficient to study  $\hat{\mu}$ over a strict subset of $\mathbb{R}^d$, e.g.,  over $[0,2 \pi N]^d$. 
For instance, the value of the Fourier-based metric \fer{fm} is achieved by searching for the ``$\sup$'' operator on the bounded set $[0,2\pi N]^d$.  Since
\[
    \frac{1}{|\k|^2}\geq \frac{1}{|\k'|^2}, \quad \forall \k \in (0,2\pi N]^d, \, 
            \forall \k' \in \erre^d_+\backslash [0,2\pi N]^d
\]
and the function 
\[
    \k \rightarrow|\hat{\mu}(\k)-\hat{\nu}(\k)|
\]
is $2\pi N$-periodic, for any given constant $s >0$ the Discrete Fourier-based metric can be defined as 
\begin{equation}\label{fm2}
    d_s(\mu,\nu)=\sup_{\k \in [0,2\pi N]^d\backslash\{0\}}{\dfrac{|\hat{\mu}(\k)-\hat{\nu}(\k)|}{|\k|^s}}.
\end{equation}

\begin{definition}[Dilated Discrete Measures]\label{dildismeas}
Given a discrete measure $\mu$ over $G_N$ and $\gamma\in \erre$ such that $\gamma>0$, the $\gamma$-dilated measure $\mu_\gamma$ is
    \[
    		\mu_\gamma(\x)=\sum_{\bm y \in G_N}\mu_{\bm y} \delta(\gamma\x - \bm y).
    \]
  The Fourier transform of $\mu_\gamma$ is
    \begin{equation}\label{eq:dilation}
    		\hat{\mu}_\gamma(\k)=\sum_{\bm x \in G_N} \mu_{\bm x}e^{-\frac{i}{\gamma}\langle \bm k, \bm x \rangle} 
			= \hat{\mu}\bigg(\frac{\k}{\gamma}\bigg).
    \end{equation}
    Therefore, if $\hat{\mu}$ is  $T$-periodic, then $\hat{\mu}_\gamma$ is $\gamma T$-periodic.
Like the original metrics \fer{fm} \cite{Toscani2007}, the metric \fer{fm2} satisfies the dilation property
 \begin{equation}\label{dila}
  d_s(\mu_\gamma,\nu_\gamma)= \frac 1{\gamma^s}  d_s(\mu,\nu).
 \end{equation}
In particular, if we consider $\mu$ of the form (\ref{discretemeasure}), the Fourier transform of its $\frac{1}{N}$-dilation is $2\pi$-periodic.

%its $\frac{1}{N}$-dilated measure has a Fourier transform that is $2\pi$-periodic over both directions.
\end{definition}

We recall the definition of the metrics \fer{baring}:
\[
	D_{s,p}(\mu,\nu):=\left(\int_{\R^d}\dfrac{|\widehat{\mu}(\k)-\widehat{\nu}(\k)|^p}{|\k|^{(sp+d)}}d\k \right)^{\frac{1}{p}},
\]
where $s>0$ and $p\geq 1$. 
As we did for the Fourier Based Metrics $d_s$, thanks to the periodicity of the Fourier transform, we can restrict the domain of integration to $[0,T]^d$.
In this case, for any given choice of the parameters $p$ and $s$, this distance is well-defined any time the integrand is integrable in a neighbourhood of the origin. This corresponds to requiring that $\frac{1}{|\k|^\gamma}$ is integrable on the $d$-dimensional ball 
$B_{1}(0)=\{\vett k\in\R^d : |\vett k|\leq 1 \}$, that is,  if and only if $\gamma<d$.
This consideration suggests the following definition.

\begin{definition}[The Periodic Fourier-based Metric]
  Let $\mu$ and $\nu$ be two probability measures over $G_N$. The $(s,p,\alpha)$-Periodic Fourier-based Metric (or PFM) between $\mu$ and $\nu$ is defined as
    \begin{equation}
    \label{PFM}
    		\fspa (\mu,\nu):=\bigg( \dfrac{1}{|T|^d}
			\int_{[0,T]^d}  \dfrac{|\widehat\mu(\k)-\widehat\nu(\k)|^p}{|\k|^{sp+\alpha}}d\k\bigg)^{\frac{1}{p}},
    \end{equation}
    where $p,s,\alpha \in \R$ and $T$ is the period of $\hat{\mu}$ and $\hat{\nu}$. When $\alpha=0$ and $s\in \mathbb{N}$ we say that $f_{s,p}:=f_{s,p}^{(0)}$ is {\it pure}.
\end{definition}

As discussed in the introduction, in dimension $d=1$ the continuous version of the metrics \fer{PFM} has been considered in \cite{Baringhaus1997}. Recently, these metrics have been considered in relation with the problem of convergence toward equilibrium of a Fokker--Planck type equation modeling wealth distribution \cite{torregrossa2017wealth}, where various properties of these metrics have been studied. 
As pointed out in \cite{torregrossa2017wealth}, if $\mu$ and $\nu$ have equal $r$-moments, the function $|\hat{\mu}(\k)-\hat{\nu}(\k)|$ behaves like $|\k|^{r+1}$ as $\k \to 0$. As a consequence, the value of $\fspa(\mu,\nu)$ is finite only if the following condition is verified
\begin{equation}
\label{cond}
	p(s-r -1)+\alpha<d.
\end{equation}
If $s,p$ and $\alpha$ satisfy \eqref{cond}, and thus $\fspa<+\infty$, we say that $\fspa$ is feasible.

\begin{proposition}
 Let $\mu$ and $\nu$ be two probability measures over $G_N$. For any given constant  $\gamma >0$, the following dilation property holds
    \begin{equation*}
    \fspa (\mu_\gamma,\nu_\gamma)=\dfrac{1}{|\gamma|^{s+\frac{\alpha}{p}}}\fspa(\mu,\nu).
    \end{equation*}
\end{proposition}

\begin{proof}
    Using relation \eqref{eq:dilation} and the change of variables $\k = \gamma \bm k'$, we get
        \begin{eqnarray*}
    \fspa(\mu_\gamma,\nu_\gamma)
    &=& \Bigg(\intaTnormd \dfrac{|\hat{\mu}_\gamma(\k)-\hat{\nu}_\gamma(\k)|^p}{|\k|^{sp+\alpha}}d\k \Bigg)^{\frac{1}{p}} \\
    &=& \Bigg(\intaTnormd  \dfrac{|\hat{\mu}(\frac{\k}{\gamma})-\hat{\nu}(\frac{\k}{\gamma})|^p}{|\k|^{sp+\alpha}}d\k \Bigg)^{\frac{1}{p}}\\
    &=&\Bigg(\dfrac{1}{|\gamma|^d}\intaTnormdued \dfrac{|\hat{\mu}(\k')-\hat{\nu}(\k')|^p}{|\gamma|^{sp+\alpha}|\k'|^{sp+\alpha}}|\gamma|^d d\k' \Bigg)^{\frac{1}{p}}\\
    &=&\dfrac{1}{|\gamma|^{s+\frac{\alpha}{p}}}\Bigg(\intaTnormdued \dfrac{|\hat{\mu}(\k')-\hat{\nu}(\k')|^p}{|\k'|^{sp+\alpha}}d\k' \Bigg)^{\frac{1}{p}}\\
    &=& \dfrac{1}{|\gamma|^{s+\frac{\alpha}{p}}}\fspa(\mu,\nu).
    \end{eqnarray*}
    {}
\end{proof}
It is important to remark that, at difference with the metrics \fer{baring}, the analogous of the dilation property \fer{dila} is true only for $\alpha=0$, that is only for pure metrics. 
We show next that the $\fspa$ metrics satisfy various monotonicity properties with respect to the parameters $p$ and $s$.

\begin{proposition}
\label{ordinos}
  Let $\mu$ and $\nu$ be two probability measures over $G_N$, with moments equal up to $r$. If $t\leq s$, then
    \begin{equation*}
    f_{t,p}^{(\alpha)}(\mu,\nu)\leq (\sqrt{d}|T|)^{(s-t)} \fspa(\mu,\nu),
    \end{equation*}
    for any $p$ and $\alpha$ for which the metric is feasible, i.e., for $p(s-r-1)+\alpha < d$.
\end{proposition}

\begin{proof}
We compute
    \begin{eqnarray*}
     f_{t,p}^{(\alpha)}(\mu,\nu) &=& \Bigg(\intTnormd \dfrac{|\hat{\mu}(\k)-\hat{\nu}(\k)|^p}{|\k|^{tp+\alpha}}d\k \Bigg)^{\frac{1}{p}}\\
    &=&\Bigg(\intTnormd \dfrac{|\k|^{p(s-t)}}{|\k|^{p(s-t)}}\dfrac{|\hat{\mu}(\k)-\hat{\nu}(\k)|^p}{|\k|^{tp+\alpha}}d\k \Bigg)^{\frac{1}{p}}\\
    &=&\Bigg(\intTnormd|\k|^{p(s-t)} \dfrac{|\hat{\mu}(\k)-\hat{\nu}(\k)|^p}{|\k|^{sp+\alpha}}d\k \Bigg)^{\frac{1}{p}}\\
    &\leq&(\sqrt{d}|T|)^{(s-t)}\fspa(\mu,\nu).
    \end{eqnarray*}
The last inequality is obtained resorting to the bound $|\k| \leq  \sqrt{d}|T|$.
{}
\end{proof}

\begin{proposition}
  Let $\mu$ and $\nu$ be two probability measures over $G_N$. If $\alpha=0$ and $p\leq q$, then
    \begin{equation*}
    f_{s,p}(\mu,\nu)\leq f_{s,q}(\mu,\nu).
    \end{equation*}
\end{proposition}

\begin{proof}
    We have 
    \begin{eqnarray*}
    f_{s,p}(\mu,\nu)&=&\Bigg(\intTnormd \dfrac{|\hat{\mu}(\k)-\hat{\nu}(\k)|^p}{|\k|^{sp}}d\k \Bigg)^{\frac{1}{p}}\\
    &=&\Bigg(\Bigg(\intTnormd \dfrac{|\hat{\mu}(\k)-\hat{\nu}(\k)|^p}{|\k|^{sp}}d\k \Bigg)^{\frac{q}{p}}\Bigg)^{\frac{1}{q}}\\
    &\leq& \Bigg(\intTnormd \Bigg(\dfrac{|\hat{\mu}(\k)-\hat{\nu}(\k)|^p}{|\k|^{sp}}\Bigg)^{\frac{q}{p}}d\k \Bigg)^{\frac{1}{q}}\\
    &=&f_{s,q}(\mu,\nu).
    \end{eqnarray*}
 The last inequality follows from Jensen's inequality.
    {}
\end{proof}
\begin{remark}
\label{lpstyle}
    %Reasoning as in the case of $L^p([0,T]^2)$ spaces, 
    By letting $p\to +\infty$, we get 
    \begin{equation*}
    \lim_{p\to \infty}f_{s,p}(\mu,\nu)=f_{s,\infty}(\mu,\nu):=d_s(\mu,\nu).
    \end{equation*}
    Thanks to H\"older inequality, for all $p<+\infty$ we have the bound
    \begin{equation}\label{bbc}
    f_{s,p}(\mu,\nu)\leq d_s(\mu,\nu).
    \end{equation}
\end{remark}
The results of this Section are preliminary to our main result, which deals with the equivalence of the pure metrics, for $p=2$,  with the Wasserstein metrics.
For the sake of simplicity, and without loss of generality, in the next subsection we consider  measures  in dimension $d=2$.

%%%%%%%%%%%%%%%%%%%%%%%%%%%%%%%%%%%%%%%%%%%%%%%%%%%%%%%%%%%%%%%%%%%%%%%%%%%%%%%%%
\subsection{\emph{Equivalence with the Wasserstein metric $W_1$}}

We consider the two cases $s=1$ and $s=2$, in dimension $d=2$, and we show that $f_{1,2}$ and $f_{2,2}$ are equivalent to $W_1$ and $W_2$, respectively.

We start with the case $s=1$. For any  $\mu,\nu \in \mathcal{P}(G_N)$,  the PFM is
\begin{equation}
\label{eqf11}
f_{1,2}(\mu,\nu)=\bigg(\intaTnormdue \dfrac{|\hat{\mu}(\k)-\hat{\nu}(\k)|^2}{|\k|^2}d\k\bigg)^{\frac{1}{2}}.
\end{equation}
We have the following 
\begin{theorem}\label{th3}
 For any pair of measures $\mu,\nu\in \mathcal{P}(G_N)$, we have the inequality
\[
     f_{1,2} (\mu,\nu) \le  W_1(\mu,\nu).
 \]
\end{theorem}
\begin{proof}

Let  $\pi$ be a transport plan between $\mu$ and $\nu$. It holds
\begin{eqnarray*}
|\mi-\ni|&=&\bigg| \sum_{\x,\y \in G_N}e^{-i\k\cdot\x}\pi(\x,\y) - \sum_{\x,\y \in G_N}e^{-i\k\cdot\y}\pi(\x,\y)\bigg|\\
&=&\bigg| \sum_{\x,\y \in G_N}\big(e^{-i\k\cdot\x}-e^{-i\k\cdot\y}\big)\pi(\x,\y) \bigg|\\
&\leq&  
\sum_{\x,\y \in G_N} \big| e^{-i\k\cdot\x}-e^{-i\k\cdot\y} \big| \pi(\x,\y)  \\
&=&  
\sum_{\x,\y \in G_N} \big| 1 - e^{i\k\cdot(\x-\y)} \big| \pi(\x,\y)  \\
&\leq&  
\sum_{\x,\y \in G_N}\big| \k\cdot(\x-\y) \big| \pi(\x,\y)\\
&\leq& |\k|\sum_{\x,\y \in G_N}|\x-\y|\pi(\x,\y).
\end{eqnarray*}
Hence, if $\pi$ is the optimal transport plan, we conclude with the inequality
\begin{equation}
|\mi-\ni|\leq |\k|W_1(\mu,\nu).
\label{thisqui}
\end{equation}
Using inequality (\ref{thisqui}) into  definition (\ref{eqf11}), we finally obtain the bound
\begin{equation}
f_{1,2}(\mu,\nu)\leq \bigg( \intaTnormdue \dfrac{\big(|\k|W_1(\mu,\nu) \big)^2}{|\k|^2}d\k \bigg)^{\frac{1}{2}}=W_1(\mu,\nu).
\label{disuguaglianzaw1}
\end{equation}
{}
\end{proof}
Since $W_1(\mu,\nu)<+\infty$ for every $\mu,\nu\in\mathcal{P}(G_N)$,  inequality (\ref{disuguaglianzaw1}) implies that $f_{1,2}$ is bounded in correspondence to any pair of probability measures over the grid $G_N$. %In addition, inequality \eqref{disuguaglianzaw1} shows that the topology induced by $f_{1,2}$ is \tre{at least} weaker than the topology induced by $W_1$.

We now show that $f_{1,2}$ and $W_1$ satisfy a reverse inequality, thus concluding that the two metrics are equivalent. %We prove
\begin{theorem}\label{th4}
 For any pair of measures $\mu,\nu\in \mathcal{P}(G_N)$ it holds
    \begin{equation}\label{rev}
    W_1(\mu,\nu)\leq  \dfrac{T^2}{2\pi} f_{1,2} (\mu,\nu).
    \end{equation}
\end{theorem}

\begin{proof}
    Owing to the dual characterization of the $W_1$ distance (see \cite{villani2008oan}, Chapter 5), there exists a $1$-\lip\  function $\phi$ such that
    \begin{equation*}
    W_1(\mu,\nu)=\int_{\R^2}\phi(\x) d\mu(\x) - \int_{\R^2}\phi(\x) d\nu(\x).
    \end{equation*}
    Since $\mu$ and $\nu$ are discrete and supported on a subset of $[0,1]^2$, we can write
    \begin{equation*}
    W_1(\mu,\nu)=\sum_{\bm x \in G_N}\phi(\x)\big(\mu_{\x}-\nu_{\x}\big).
    \end{equation*}
    Therefore, resorting to the fact that both the measures have the same mass, for any given constant $c\in \erre$ we have
    \begin{equation*}
    W_1(\mu,\nu)=\sum_{\bm x \in G_N}\big(\phi(\x)+c\big)\big(\mu_{\x}-\nu_{\x}\big).
    \end{equation*}
    The last identity permits to choose $\phi$ such that $\phi(\frac{N}{2},\frac{N}{2})=0$. Since $\phi$ is $1$-\lip, we conclude that
    \begin{equation}\label{bb}
    |\phi(\x)|\leq \dfrac{\sqrt{2}}{2}, \qquad\qquad \forall \x\in G_N.
    \end{equation}
    By H\"older inequality we obtain
    \begin{equation*}
    W_1(\mu,\nu)\leq \bigg(\sum_{\bm x \in G_N}|\phi(\x)|^2 \bigg)^{\frac{1}{2}}\bigg(\sum_{\bm x \in G_N}|\mu_{\x}-\nu_{\x}|^2 \bigg)^{\frac{1}{2}}.
    \end{equation*}
    Since  
    \[
	\sum_{\bm x \in G_N} |\mu_{\x}-\nu_{\x}|^2=
	\intaTnormdue A(\k) B(\k) d\k
	\]
	where
	\begin{eqnarray*}
	A(\k)&=&\sum_{\x\in G_N}(\mu_{\x}-\nu_{\x})e^{-i<\x,\k>}\\
	B(\k)&=&\sum_{\y \in G_N}(\mu_{\y}-\nu_{\y})e^{+i<\y,\k>}
	\end{eqnarray*}
	we have 
	\[
	\sum_{\bm x \in G_N} |\mu_{\x}-\nu_{\x}|^2=
	\intaTnormdue\big|\mi-\ni\big|^2d\k.
	\]
    Now using \eqref{bb} we obtain
    %which can be rewritten, using Plancherel identity and \fer{bb} as
    \begin{eqnarray*}
    W_1(\mu,\nu)&\leq& \dfrac{\sqrt{2}N}{2}\bigg(\intaTnormdue\big|\mi-\ni\big|^2d\k \bigg)^{\frac{1}{2}}\\ &=&\dfrac{\sqrt{2}N}{2}\bigg(\intaTnormdue|\k|^2\dfrac{|\hat\mu(\k)-\hat\nu(\k)|^2}{|\k|^2}d\k \bigg)^{\frac{1}{2}}.
    \end{eqnarray*}
    Since $|{\bm k}|^2\leq 2T^2$ and $T=2\pi N$, we can finally conclude that
    \begin{equation*}
    W_1(\mu,\nu)\leq \dfrac{T^2}{2\pi} \bigg(\intaTnormdue\dfrac{|\mi-\ni|^2}{|\k|^2}d\k \bigg)^{\frac{1}{2}}=\dfrac{T^2}{2\pi} f_{1,2} (\mu,\nu).
    \end{equation*}
    {}
\end{proof}

In consequence of the previous estimates, it is immediate to show that the metrics $d_s$ and $W_1$ are equivalent. This is proven in the following

\begin{corollary}\label{cor1}
For any pair of measures $\mu,\nu\in \mathcal{P}(G_N)$ 
   \[
   		d_1 (\mu,\nu)\le W_1(\mu,\nu)\leq \dfrac{T^2}{2\pi} d_1 (\mu,\nu).
   \]
\end{corollary}

\begin{proof}

The first inequality is a consequence of bound \fer{thisqui}. The second one follows from inequality \fer{bbc}. 
    {}
\end{proof}
%%%%%%%%%%%%%%%%%%%%%%%%%%%%%%%%%%%%%%%%%%%%%%%%%%%%%%%%%%%%%%%%%%%%%%%%%%%%%%%%%
\subsection{\emph{Equivalence with the Wasserstein metric $W_2$}}

The aim of this Section is to show the equivalence of the Fourier-based metric $f_{2,2}$ and the Wasserstein metric $W_2$. 
Let $s=2$. In this case, the PFM takes the form
\[    
	f_{2,2}(\mu,\nu)=\bigg(\intaTnormdue\dfrac{|\mi-\ni|^2}{|\k|^4}d\k \bigg)^{\frac{1}{2}}.
\]
Clearly, the distance between the two probability measures is well-defined only when $\mu$ and $\nu$ possess the same expected value. Since, in general this is not the case, we start by translating the measures, as done in Section \ref{recallot}, in order to satisfy this condition.
The following proposition shows that, for probability measures with the same center, %mean value, 
the topology induced by $f_{2,2}$ is not stronger than the topology induced by $W_2$.

\begin{theorem}
\label{mezzaeq22}
For any pair of measures $\mu,\nu\in \mathcal{P}(G_N)$  such that $\mmu=\nnu$,
 it holds
    \begin{equation}\label{dd}
    f_{2,2}(\mu,\nu)\leq 2\sqrt2 W_2(\mu,\nu).
    \end{equation}
    In particular, $f_{2,2}(\mu,\nu)<\infty.$
\end{theorem}{}
\begin{proof}

For any given pair of probability measures $\mu$ and $\nu$ in $\mathcal{P}(G_N)$, with centers $\mmu=\nnu$, we have
\begin{equation*}
i\k \sum_{\x\in G_N}\x \mu_{\x} = i \k \sum_{\y\in G_N} \y\nu_{\y}.
\end{equation*}
For any transport plan $\pi$ between $\mu$ and $\nu$, we can rewrite the previous relations in the form
\begin{equation}\label{cc}
i\k\sum_{\x,\y\in G_N}(\x-\y)\pi_{\x,\y}=0.
\end{equation}
Using identity \fer{cc} we obtain
\begin{eqnarray*}
	\hat{\mu}(\k)-\hat{\nu}(\k)&=& \sum_{\x\in G_N}  \mu_{\x}  e^{-i\k \cdot \x}- \sum_{\y\in G_N}  \nu_{\y} e^{-i\k \cdot \y}\\
&=& \sum_{\x,\y\in G_N} \bigg( e^{-i\k \cdot \x}-e^{-i\k\cdot \y}- i\k\cdot(\x-\y) \bigg)\pi_{\x,\y}   \\
&=&  \sum_{\x,\y\in G_N} e^{-i\k\cdot \y} \big( e^{-i\k \cdot(\x-\y)}-1-i\k\cdot(\x-\y) \big)\pi_{\x,\y}\\
&\quad&+  \sum_{\x,\y\in G_N} i\k\cdot(\x-\y)(e^{-i\k\cdot\y}-1)\pi_{\x,\y}.
\end{eqnarray*}
Using that for all $\theta \in \R$
\begin{align*}
	|e^{i\theta}-1|&\leq |\theta|, \\
	|e^{i\theta}-1-i\theta|&\leq \frac{\theta^2}2
\end{align*}	
we obtain
\begin{eqnarray*}
|\hat{\mu}(\k)-\hat{\nu}(\k)|&\leq& \frac{|\k|^2}{2}\sum_{\x,\y\in G_N} |\x-\y|^2\pi_{\x,\y} +|\k|^2 \sum_{\x,\y\in G_N} |\x-\y||\y|\pi_{\x,\y} \\
&\leq& \frac{|\k|^2}{2}\sum_{\x,\y\in G_N} |\x-\y|^2\pi_{\x,\y} \\
&\quad& +|\k|^2\bigg( \sum_{\x,\y\in G_N} |\y|^2\pi_{\x,\y} \bigg)^{\frac{1}{2}}\bigg( \sum_{\x,\y\in G_N} |\x-\y|^2\pi_{\x,\y}\bigg)^{\frac{1}{2}}.\end{eqnarray*}
In particular, if we take $\pi$ as the optimal transportation plan between $\mu$ and $\nu$ for the cost $|\x-\y|^2$ we get
\begin{align*}
	\frac{|\hat{\mu}(\k)-\hat{\nu}(\k)|}{|\k|^2}&\leq \frac{W_2^2(\mu,\nu)}{2}+ 
		\bigg( \sum_{\y\in G_N}|\y|^2\nu_{\y} \bigg)^{\frac{1}{2}}W_2(\mu,\nu)\\
		& =W_2(\mu,\nu) \left( \frac{W_2(\mu,\nu)}{2} + \left( \sum_{\y\in G_N}|\y|^2\nu_{\y} \right)^{\frac{1}{2}}\right).
\end{align*}
Since 
\[
	W_2(\mu,\nu)\leq W_2(\mu,\delta) + W_2(\delta,\nu) \leq \left(\sum_{\x\in G_N}|\x|^2\mu_{\x}\right)^{\frac{1}{2}} + \left(\sum_{\y\in G_N}|\y|^2\nu_{\y}\right)^{\frac{1}{2}},
\]
and,
%\begin{equation}
%%\label{w2tosc}
%	\frac{|\hat{\mu}(\k)-\hat{\nu}(\k)|}{|\k|^2}\leq \left(\left(\sum_{\x\in G_N}|\x|^2\mu_{\x}\right)^{\frac{1}{2}} + \left(\sum_{\y\in G_N}|\y|^2\nu_{\y}\right)^{\frac{1}{2}}\right)W_2(\mu,\nu).
%\end{equation}
 as $\mu$ and $\nu$ are supported in  $[0,1]^2$,
 \[
  \sqrt{\sum_{\x\in G_N}|\x|^2 \mu_{\x}} \leq \sqrt 2,\quad \sqrt{\sum_{\y\in G_N}|\y|^2 \nu_{\y}} \le \sqrt2,
 \]
we obtain \fer{dd}:
 \[
 \dfrac{|\mi-\ni|}{|\k|^2} \le 2\sqrt2 W_2(\mu,\nu).
 \]
 
\end{proof}

We conclude by showing the validity of a reverse inequality,  thus proving the equivalence between $f_{2,2}$ and $W_2$.

\begin{theorem}
\label{altramezza22}
 For any pair of measures $\mu,\nu\in \mathcal{P}(G_N)$, we have the inequality
    \[
    W_2^2(\mu,\nu)\leq \dfrac{T^3}{\pi} f_{2,2}(\mu,\nu).
    \]
\end{theorem}

\begin{proof}
Let $\pi$ be the optimal transportation plan between $\mu$ and $\nu$ for the cost $|\x-\y|$, since $|\x-\y|\leq \sqrt 2$ for all $\x,\y\in G_N\subset [0,1]^2$, it holds
\[
    W_2^2(\mu,\nu)\leq \sum_{\x,\y \in G_N}|\x - \y|^2\pi_{\x,\y} 
    	\leq \sum_{\x,\y \in G_N}\sqrt 2|\x - \y|\pi_{\x,\y} = \sqrt{2}W_1(\mu,\nu).
\]
%
%When $\mu$ and $\nu$ are supported on $[0,1]^2$, we have
%    \[
%    W_2^2(\mu,\nu)\leq \sqrt{2}W_1(\mu,\nu).
%    \]
    Then, by Theorem \ref{th4} and Proposition \ref{ordinos} with $t=1$ and $p=s=2$, we get
\[
     \sqrt2 W_1(\mu,\nu)  \leq \dfrac{\sqrt2 T^2}{2\pi} f_{1,2}(\mu,\nu) \leq  \dfrac{T^3}{\pi} f_{2,2}(\mu,\nu),
\]
which, together with the last inequality, concludes the proof.
%    \begin{equation*}
%     W_2^2(\mu,\nu)\leq \sqrt{2}W_1(\mu,\nu) 
%     \leq \sqrt{2}\dfrac{T^2}{2\pi} f_{1,2}(\mu,\nu) \leq  \dfrac{T^3}{\pi} f_{2,2}(\mu,\nu)
%    \end{equation*}
    
\end{proof}

The previous bounds hold provided that $\mu$ and $\nu$ are centered in the same point.
However, when $\mmu-\nnu\neq 0$, we can resort, as in Section \ref{recallot}, to the new metric  
    \[
    \mathcal{F}_{2,2}(\mu,\nu) := \sqrt{\left(f_{2,2}(\mu,\nu_{\mmu-\nnu})^2+|\mmu-\nnu|^2\right)},
    \]
which is well-defined also for probability measures having  different centers. This shows that we can generalize, similarly to  Theorem \ref{teoremadistanza} and Theorem \ref{teoremaequivalenza}, the equivalence of $\mathcal{F}_{2,2}$ and $W_2$ to measures which are not centered in the same point.

\subsection{\emph{Connections with other distances}}
As discussed in \cite{torregrossa2017wealth}, the case in which $s \le 0$ leads to stronger metrics. In this case, we clearly loose  relations like \fer{dd}, that link from above the Wasserstein metric with the Fourier-based metric. An interesting case is furnished by choosing both $s=0$ and $\alpha =0$ into \fer{PFM}. The metric in this case is defined by
    \begin{eqnarray*}
    f_{0,2}(\mu,\nu)&=&\bigg(\intnormpiN |\hat{\mu}(\k)-\hat{\nu}(\k)|^2d\k\bigg)^{\frac{1}{2}}
    \\
    &=&\bigg(\sum_{\x\in G}|\mu(\x)-\nu(\x)|^2\bigg)^{\frac{1}{2}},
    \end{eqnarray*}
which is  the Total Variation distance between the probability measures $\mu$ and $\nu$. 
    
We remark that  the distance above does not require the measures to possess the same mass.  By fixing in  definition \fer{PFM} $s=0$ and $\alpha\in [0,2)$, one obtains a  sequence of metrics that interpolate between the Total Variation distance and the $W_1$ distance, namely a family of measures that move from a strong metric to a weaker one. However, if $\alpha>0$, the measures must have the same mass.
    
In the case $s<0$ and $\alpha =0$, the Fourier-based metric  \eqref{PFM} becomes
    \begin{equation*}
    f_{s,2}(\mu,\nu)=\bigg(\intnormpiN |\k|^{2|s|}|\hat{\mu}(\k)-\hat{\nu}(\k)|^2d\k \bigg)^{\frac{1}{2}}.
    \end{equation*}
    In particular, when $-s=n \in \mathbb{N}_+$, we find that
    \begin{equation*}
    f_{-n,2}(\mu,\nu)=\bigg(\intnormpiN |\k|^{2n}|\hat{\mu}(\k)-\hat{\nu}(\k)|^2d\k \bigg)^{\frac{1}{2}}.
    \end{equation*}
This metric, by Fourier identity, controls the $n$-th derivative of the measures $\mu$ and $\nu$, and does not require the measures to have the same mass.

\section{Numerical Results}
\label{Numericalresults}
We run extensive numerical tests to compare the Wasserstein metrics $W_1$ and $W_2$ with the corresponding Periodic Fourier-based Metrics $f_{1,2}^0$ and $f_{2,2}^0$. 

The goal of our tests is to compare empirically the distance values obtained with the different metrics, and to measure the runtime gain that we can achieve using the Fourier-based metrics.
In the following paragraphs, we report the main conclusions of our tests.

\paragraph{Implementation details.}
We implemented our algorithms in Python 3.7, using the Fast Fourier Transform implemented in the Numpy library \cite{numpy}. To compute the Wasserstein distances, we use the Python Optimal Transport (POT) library \cite{flamary2017pot}. All the tests are executed on a MacBook Pro 13 equipped with a 2.5 GHz Intel Core i7 dual-core and 16 GB of Ram.

\paragraph{Dataset.}
As problem instances, we use the DOTmark benchmark \cite{Dotmark}, which contains 10 classes of gray scale images, each containing 10 different images. Every image is given in the data set at the following pixel resolutions: $32 \times 32, 64 \times 64, 128 \times 128, 256 \times 256$, and $512 \times 512$. Figure \ref{fig:dot} shows the Classic, Microscopy, and Shapes images, respectively, at the highest pixel resolution (one class for each row).

\begin{figure}[!ht]
\centering
{\renewcommand{\arraystretch}{0.7}
\setlength{\tabcolsep}{0.1em}
\begin{tabular}{cccccccccc}
\includegraphics[width=0.095\linewidth]{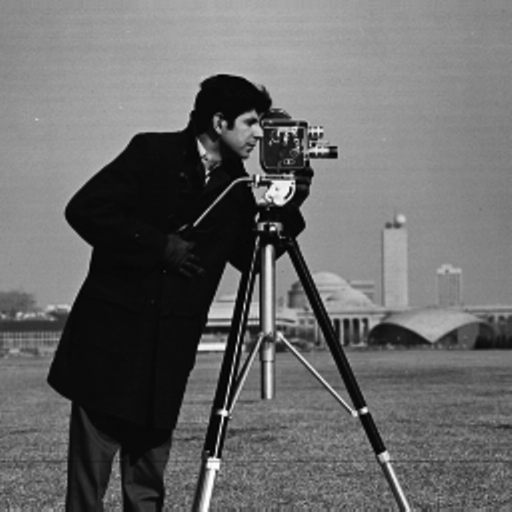}  & \includegraphics[width=0.095\linewidth]{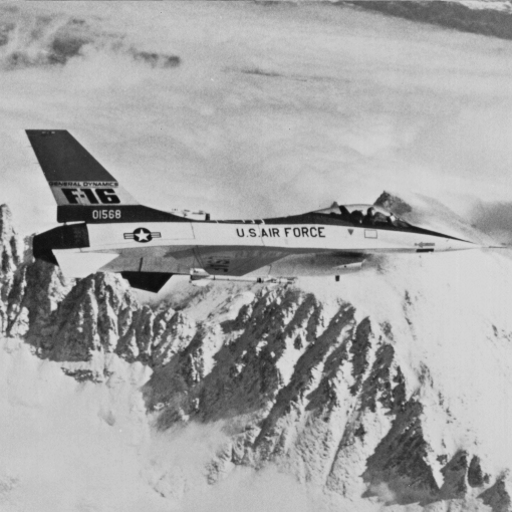} &
  \includegraphics[width=0.095\linewidth]{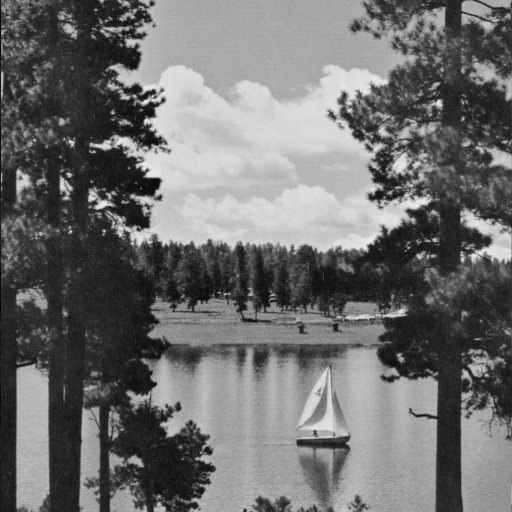} & \includegraphics[width=0.095\linewidth]{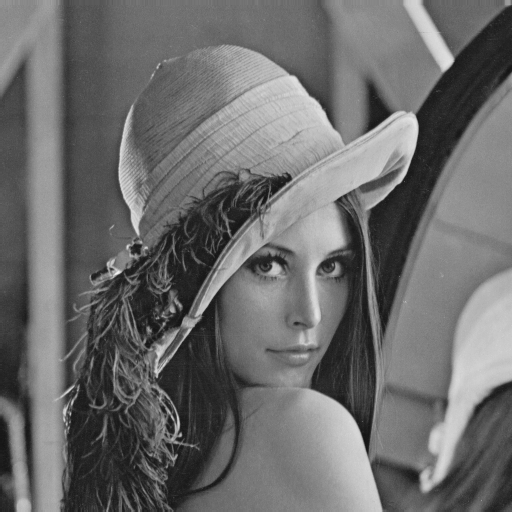} &
  \includegraphics[width=0.095\linewidth]{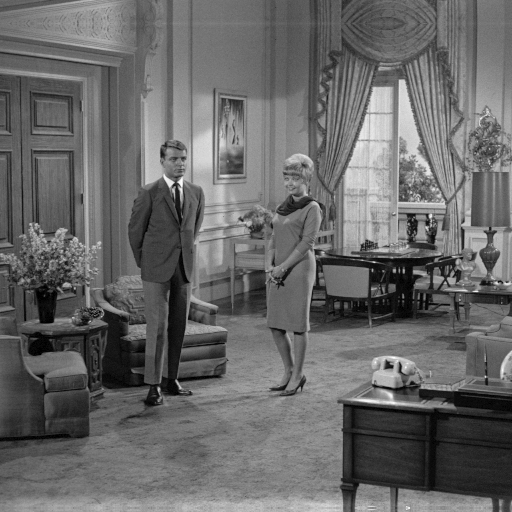} &
  \includegraphics[width=0.095\linewidth]{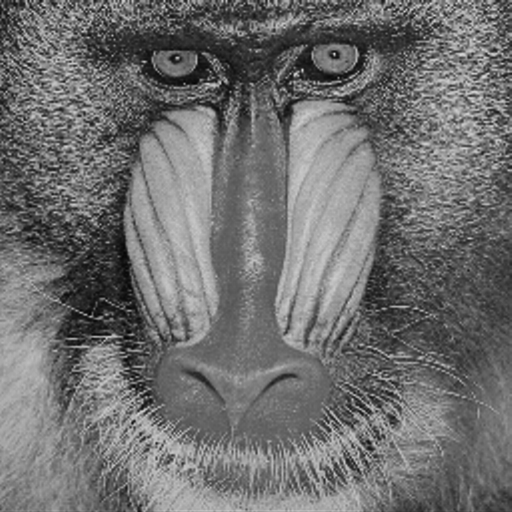}  & \includegraphics[width=0.095\linewidth]{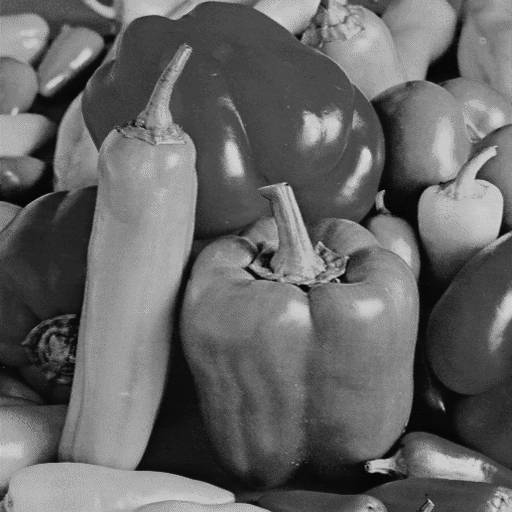} &
  \includegraphics[width=0.095\linewidth]{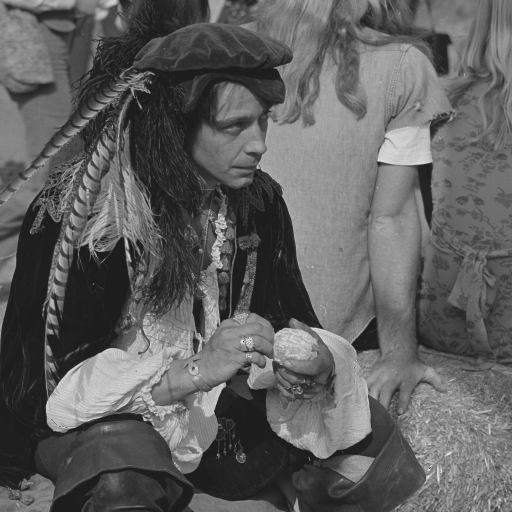} & \includegraphics[width=0.095\linewidth]{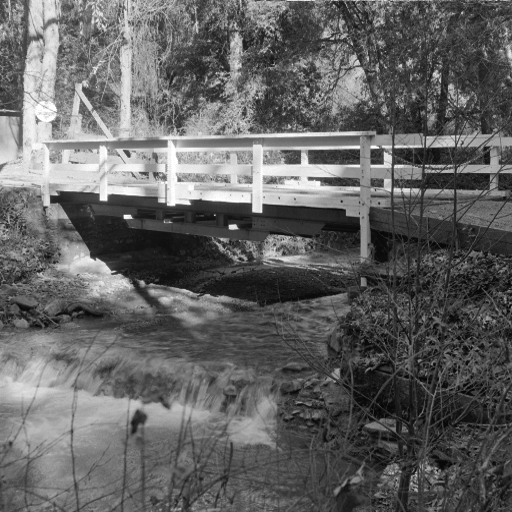} &
  \includegraphics[width=0.095\linewidth]{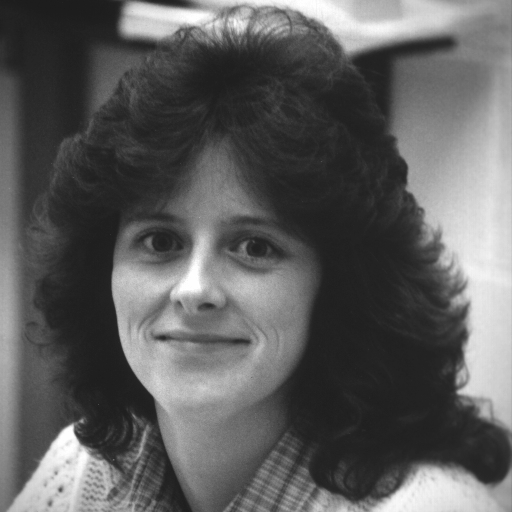} \\  
 \includegraphics[width=0.095\linewidth]{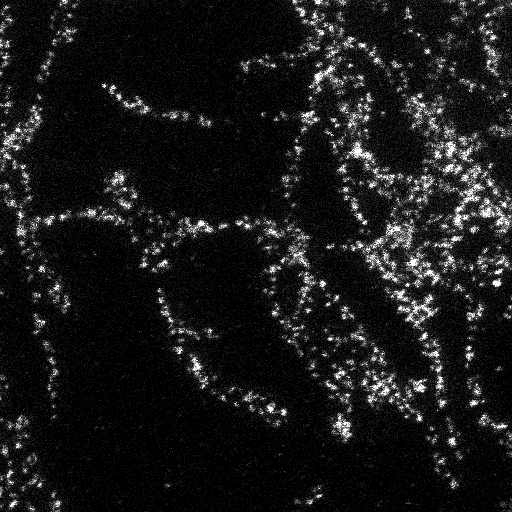}  & \includegraphics[width=0.095\linewidth]{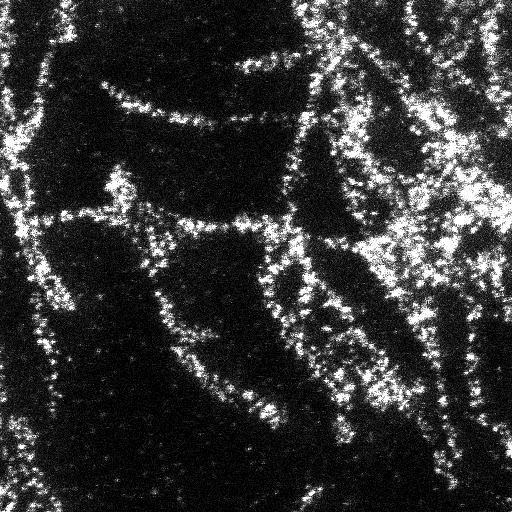} &
  \includegraphics[width=0.095\linewidth]{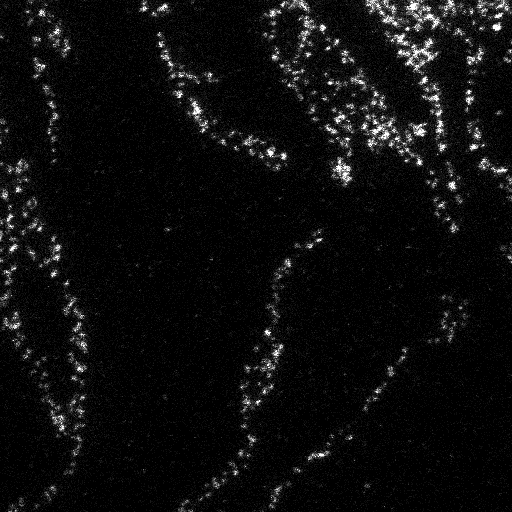} & \includegraphics[width=0.095\linewidth]{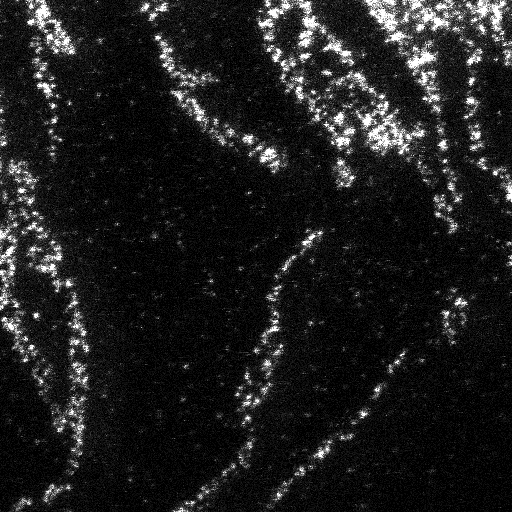} &
  \includegraphics[width=0.095\linewidth]{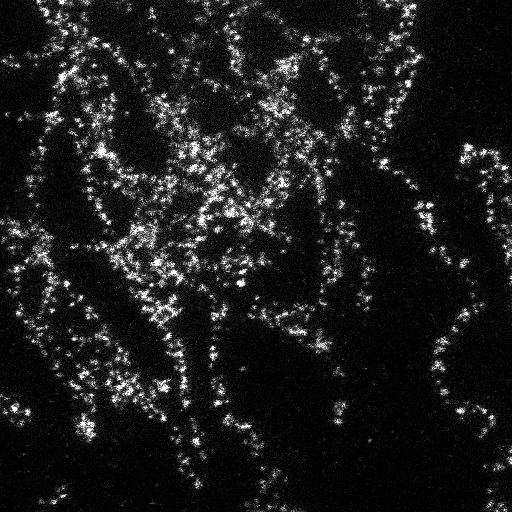} &
  \includegraphics[width=0.095\linewidth]{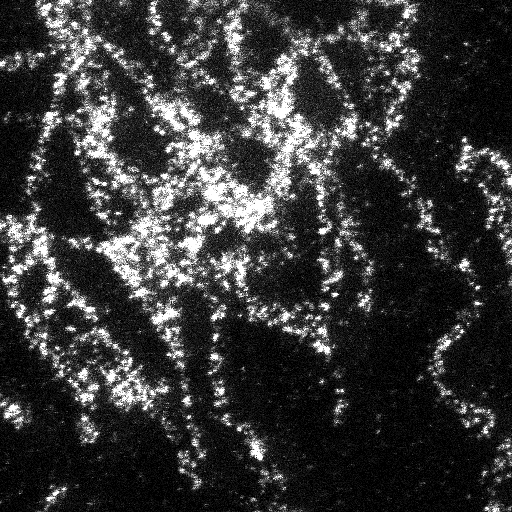}  & \includegraphics[width=0.095\linewidth]{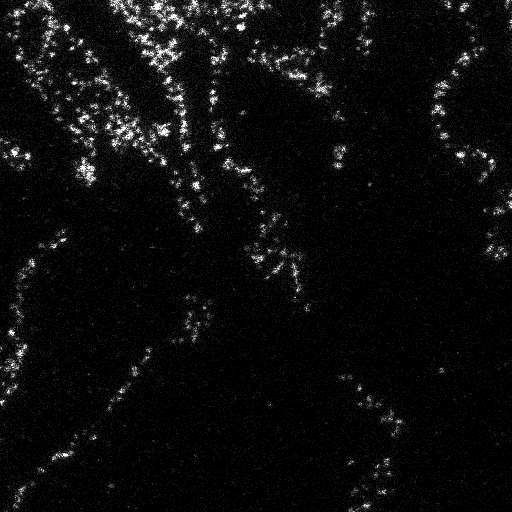} &
  \includegraphics[width=0.095\linewidth]{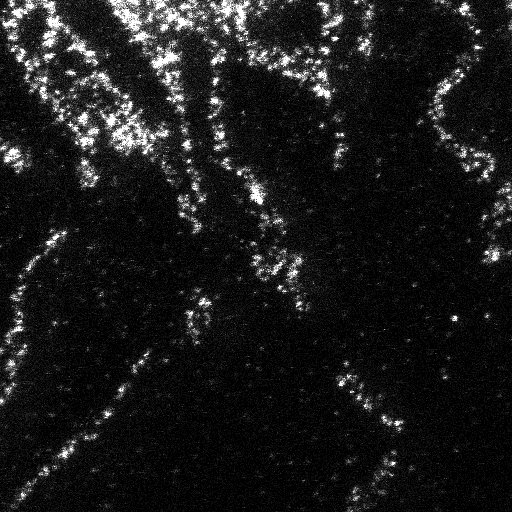} & \includegraphics[width=0.095\linewidth]{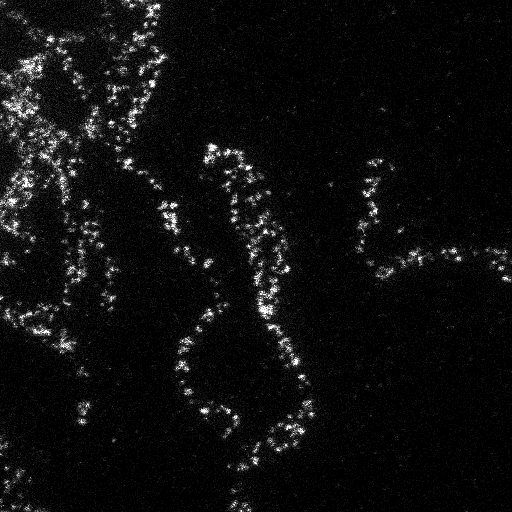} &
  \includegraphics[width=0.095\linewidth]{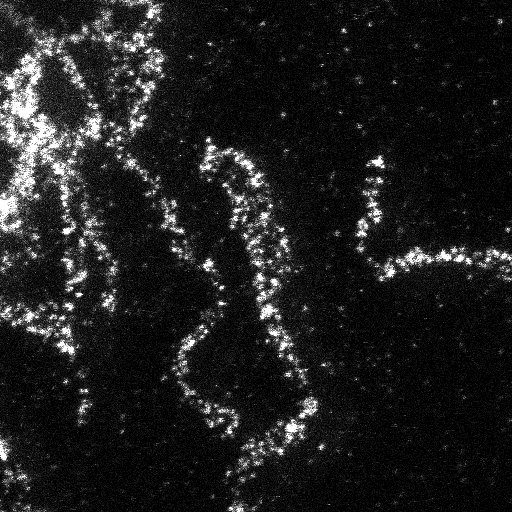} \\
 \includegraphics[width=0.095\linewidth]{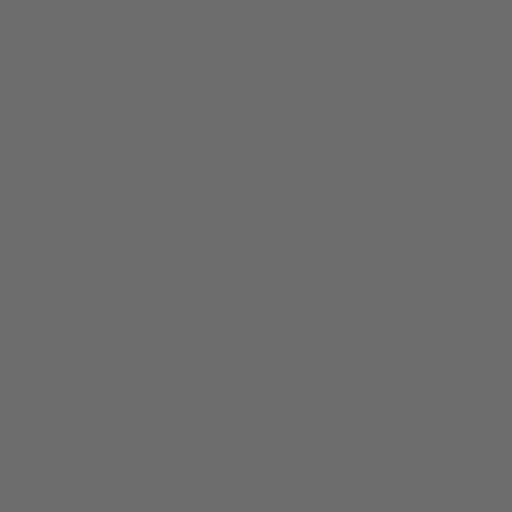}  & \includegraphics[width=0.095\linewidth]{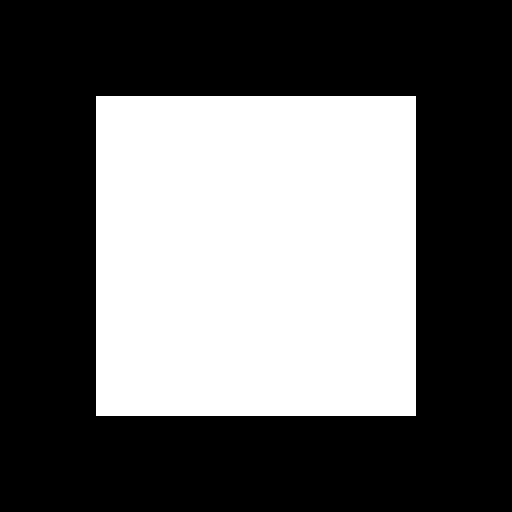} &
  \includegraphics[width=0.095\linewidth]{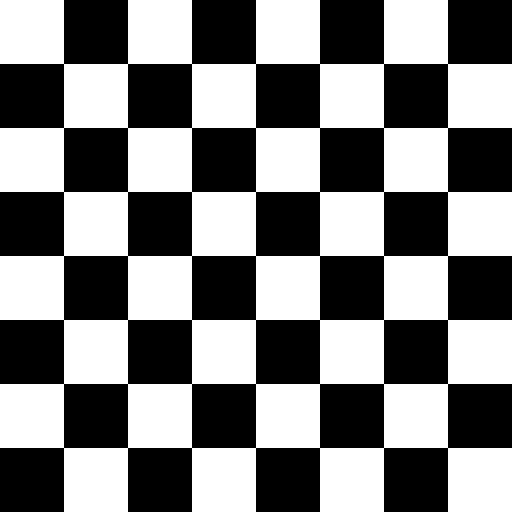} & \includegraphics[width=0.095\linewidth]{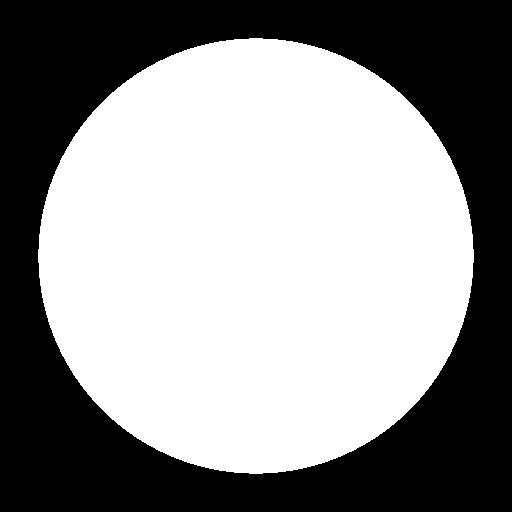} &
  \includegraphics[width=0.095\linewidth]{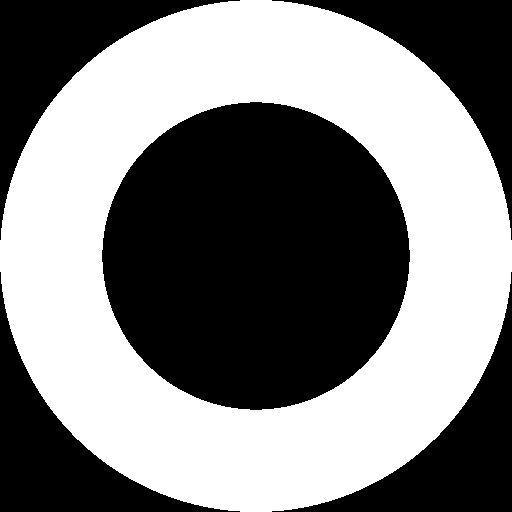} &
  \includegraphics[width=0.095\linewidth]{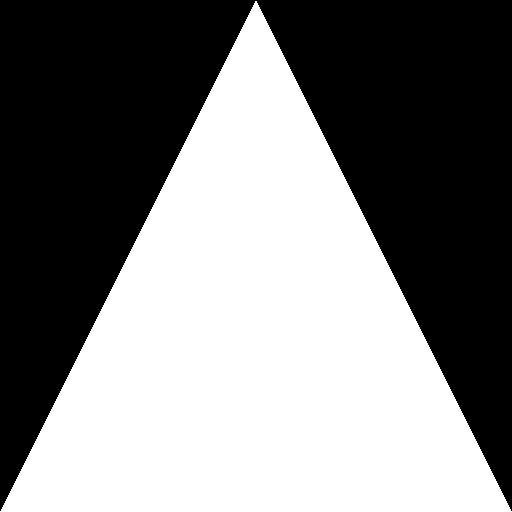}  & \includegraphics[width=0.095\linewidth]{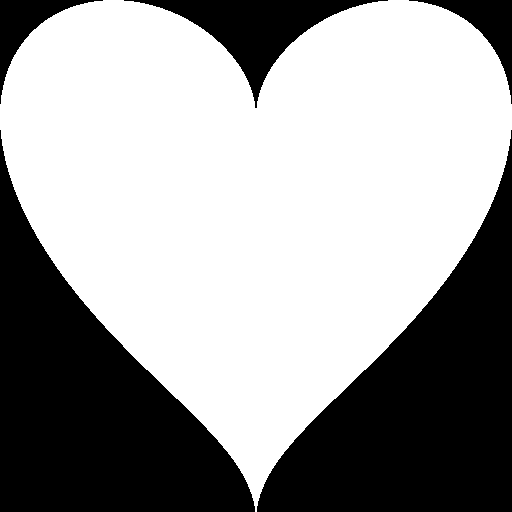} &
  \includegraphics[width=0.095\linewidth]{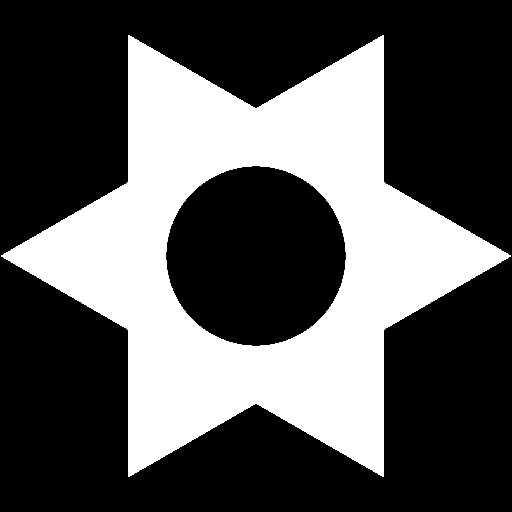} & \includegraphics[width=0.095\linewidth]{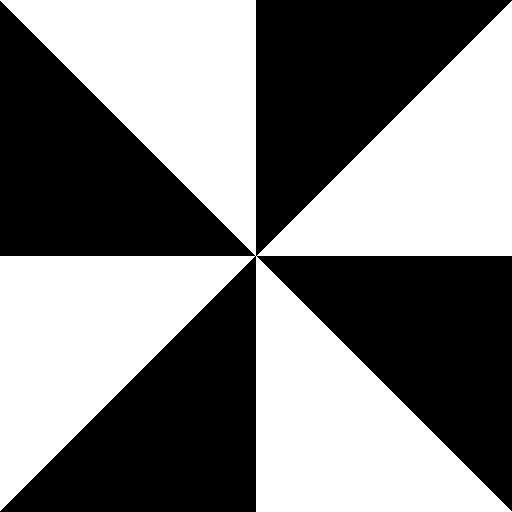} &
  \includegraphics[width=0.095\linewidth]{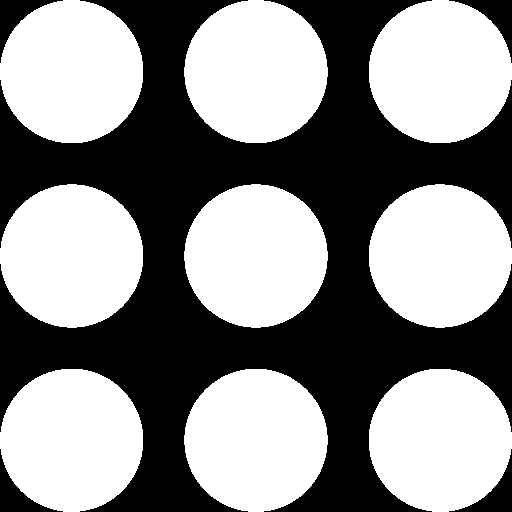} \end{tabular}}
\caption{DOTmark benchmark: Classic, Microscopy, and Shapes images. \label{fig:dot}}
\end{figure}

\paragraph{Results.}
For each pair of images of the DOTmark dataset,  the reciprocal distance values using the $W_1, W_2, f_{1,2}^{0}$  and $f_{2,2}^{0}$ metrics has been computed, and the corresponding runtime in seconds has been recorded.

The scatter plot in Figure \ref{lwbond} shows the relation between the $W_2$ and the $f_{2,2}^0$ distances for each pair of images at pixel resolution $32 \times 32$. The plot shows that not only the two metrics are theoretically equivalent, as proved in Theorem \ref{mezzaeq22} and \ref{altramezza22}, but also that they yield very similar values. A partial exception is present in the Shape class, which, however, contains artificial shape images. On the much more (application-wise) interesting Classic images, the two metrics return very close values.

\begin{figure}[t]
  \includegraphics[width=\linewidth]{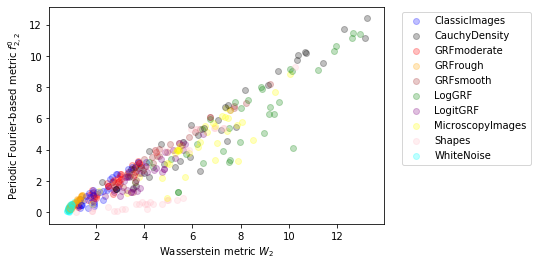}
  \caption{Wasserstein metric $W_2$ versus Periodic Fourier-based metric $f_{2,2}^0$: Comparison of distance values for 450 pair of images of size $32\times32$.\label{lwbond}}
\end{figure}

Table \ref{tab:1} reports the averages and the standard deviations of the runtime, measured in seconds, at different image size.
For each row and each metric, the averages are computed over 450 instances. The numerical results clearly show that the PFM metrics
are orders of magnitude faster, and permit to compute the distance even for the largest $512 \times 512$ images in around 10 seconds.
Note that using the POT library, we were unable to compute the $W_1$ and $W_2$ distances for images of size $256 \times 256$ and $512 \times 512$, due to memory issues.

\section{Conclusions}
\label{fine}
In this paper we showed that the class of Fourier-based metrics introduced in \cite{gabetta1995metrics} and \cite{Baringhaus1997} are useful tools to measure the distance between pairs of probability distributions, which, in reason of their equivalence,  represent an interesting alternative  in  problems where Wasserstein distances were already successfully employed.

The main result of this paper is that the constants in the equivalence relation can be precisely quantified if discrete probability measures are considered. In addition,  preliminary computational results have shown that in image processing, at difference with Wasserstein metrics,  Fourier metrics provide a noticeable performance with respect to time, even when dealing with very large images. Starting from these results, we believe it will be possible to design new numerical methods in computer imaging, which combine  theoretical convergence results with a low computational cost, at difference with Wasserstein metric, which, nowadays, has still a heavy computational load.

\begin{table}[!ht]
 \caption{Runtime vs. Image size for different metrics: The runtime is measured in seconds and reported as ``{\it Mean (StdDev)}''. Each row gives the averages over 450 instances of pairwise distances.
\label{tab:1}}
\centering
{\renewcommand{\arraystretch}{1.2}
\begin{tabular}{lr@{ }lr@{ }lr@{ }lr@{ }l}
\hline
 & \multicolumn{8}{c}{Averages Runtime in seconds} \\
Dimension & \multicolumn{2}{c}{$W_1$} & \multicolumn{2}{c}{$W_2$}  & \multicolumn{2}{c}{$f_{1,2}^{0}$}  & \multicolumn{2}{c}{$f_{2,2}^{0}$}  \\
\hline
$32\times32$    & $0.84$ & $(0.30)$ &$1.06$  & $(0.32)$ &$0.002$ &  $(10^{-4})$ &$0.006$ & $(10^{-4})$\\
$64\times64$    & $21.9$ & $(7.96)$ &$23.4$ & $(8.49)$ &$0.01$ & $(10^{-3})$ &$0.02$ & $(10^{-3})$\\
$128\times128$    & $205.0$ & $(45.9)$ & $199.0$ & $(45.0)$ & $ 0.28$ & $(0.07)$ &$0.63$ & $(0.16)$ \\
$256\times256$    & &&&   & $1.21$ & $(0.40)$ & $2.96$  & $(0.94)$ \\
$512\times512$    & &&&   & $4.74$ & $(1.32)$ & $11.55$ & $(2.84)$ \\
\hline
 \end{tabular}}

\end{table}

%A further step in this direction will consist in the application of these new Fourier-based metrics to real problems, where Wasserstein distances were already successfully employed.

\section*{Acknowledgements}
This paper was written within the activities of the GNFM of INDAM. 
The research was partially supported by the Italian Ministry of Education, University and Research (MIUR): 
Dipartimenti di Eccellenza Program (2018--2022) - Dept. of Mathematics ``F. Casorati'', University of Pavia.
 The PhD scholarship of Andrea Codegoni is founded by Sea Vision S.r.l..


%\bibliographystyle{plain}
%\bibliography{references}
%\end{document}

\end{document}